\documentclass[12pt,letterpaper]{amsart}
\usepackage[pass]{geometry}
\usepackage{amssymb,amsmath,amsthm,amsgen}
\usepackage{tikz}
\usepackage{comment}
\usepackage{}
\usepackage{enumitem}

\newcommand{\old}[1]{}
\theoremstyle{plain}
\newtheorem{thm}{Theorem}[section]
\newtheorem{lem}[thm]{Lemma}
\newtheorem{conj}{Conjecture}
\newtheorem{cor}[thm]{Corollary}

\newtheorem{prop}[thm]{Proposition}
\theoremstyle{definition}
\newtheorem{defn}[thm]{Definition}

\newtheorem{ex}[thm]{Example}
\newtheorem{rk}[thm]{Remark}
\newtheorem{qn}[thm]{Question}

\numberwithin{equation}{section}
 at 10truept

\title{Prime power variations of higher $Lie_n$ modules}

\author{Sheila Sundaram}
\address{Pierrepont School, One Sylvan Road North, Westport, CT 06880}
\email{shsund@comcast.net}

\date{\today}


\begin{document}
\begin{abstract}
We define, for each subset $S$ of the set $\mathcal{P}$  of primes, an $S_n$-module $Lie_n^S$ 
with interesting properties.  $Lie_n^\emptyset$ is the well-known representation $Lie_n$ of $S_n$ afforded by the free Lie algebra, while  $Lie_n^\mathcal{P}$ is the module $C\!onj_n$ of the conjugacy action of $S_n$ on $n$-cycles.  
 For arbitrary $S$ the module $Lie_n^{S}$ 
  interpolates between the representations $Lie_n$  and  $C\!onj_n.$  We consider the  symmetric and exterior powers of  $Lie_n^S.$ These are  the analogues of the higher Lie modules of Thrall. We show that the Frobenius characteristic of these higher $Lie_n^S$ modules can be elegantly expressed as a multiplicity-free sum of power sums. In particular this establishes the Schur positivity of new classes of  sums of power sums. 

 More generally, for each  nonempty subset $T$ of positive integers 
 we define a sequence of symmetric functions $f_n^T$ of homogeneous degree $n.$ We show that  the series $\sum_{\lambda, \lambda_i\in T} p_\lambda$ can be expressed as symmetrised powers of the functions $f_n^T$, analogous to the higher Lie modules first defined by Thrall.
 This in turn allows us to unify previous results on the Schur positivity of multiplicity-free sums of power sums, as well as investigate new ones.  We also uncover some curious plethystic relationships between $f_n^T$, the conjugacy action and the Lie representation.

\end{abstract}
\maketitle

\section{Introduction}

For each irreducible character $\chi^\lambda$ of the symmetric group of $S_n,$  indexed by a partition $\lambda$ of $n,$ 
and any subset $T$ of the conjugacy classes, one can form the sum $\sum_{\mu\in T} \chi^\lambda(\mu),$ and ask when this sum is nonnegative.   
In the language of symmetric functions, one asks for what subsets $T$ of partitions of $n$ the sum of power sums $\sum_{\mu\in T} p_\mu $ is the Frobenius characteristic of a true representation of $S_n,$ i.e. a \textit{Schur positive} symmetric function,  one with nonnegative (and in this case, necessarily integer) coefficients in the basis of Schur functions.  
A  method for generating such  classes of subsets $T$ was presented in \cite{Su1}.  This paper  generalises and extends these results with classes indexed by subsets of primes (Theorems 3.6 and 3.9).  We show that such sums of power sums are in fact realised by  analogues of the higher $Lie$ modules defined by Thrall \cite{T}, where the representation $Lie$ is replaced by an appropriate variant. In  Section~\ref{SecPrimes} we define certain $S_n$-modules $Lie_n^S$ indexed by subsets $S$ of primes, and induced  from the cyclic subgroup  generated by an $n$-cycle,  and show that their symmetric powers 
can be expressed as the product $\prod_{n\in P( S)} (1-p_n)^{-1},$  for an appropriate subset $P(S)$ determined by $S.$   This in effect gives a compact and elegant formula for the \textit{higher} $Lie_n^S$-module as a sum of power sums, thereby establishing Schur positivity of the latter. 
When $S$ consists of the single prime 2, the module $Lie_n^{(2)}$ was shown to have  properties that are remarkably analogous to virtually every known property of the representation $Lie_n$ of $S_n$ on the multilinear component of the free Lie algebra    {\cite{SuLieSup2}}. It is the discovery and construction of $Lie_n^{(2)}$  that led to the generalisations  of the present paper.  The interesting choices of $S$ other than $Lie^{(2)}$ are recorded  in   Theorems  \ref{SymAltSymLS}-\ref{Spositivity1} and  \ref{Spositivity2}-\ref{OnePrime}.  These results  were announced in \cite{SuFPSAC2018}. 

 Another approach to the Schur positivity problem for power sums is developed in Section~\ref{SecfnT}.  We give a general formula which expresses the product $\prod_{n\in T} (1-p_n)^{-1}$ as a symmetrised module, the analogue of a higher Lie module,  over a sequence of  possibly virtual representations $f_n^T,$ with the  property that their characters vanish unless the conjugacy class has all cycles of equal length.  The goal then is to determine for what choices of the set $T$ the $f_n^T$ are Frobenius characteristics of true 
$S_n$-modules, thereby establishing the Schur positivity of the 
product $\prod_{n\in T} (1-p_n)^{-1}.$    The module $Lie_n$ plays a prominent role in the construction.
In the course of these calculations many interesting plethystic identities emerge, as well as many  new conjectures on Schur positivity.  
The paper \cite{Su1} develops machinery which greatly facilitates and unifies the plethystic computations involved in finding symmetric and exterior powers, or equivalently, the higher Lie analogues, of a certain class of $S_n$-modules. These tools  are summarised in Section~\ref{SecPrelim}.
  We conclude in Section~\ref{SecPleth} with a summary of plethystic inverse identities that have frequently come up in different homological contexts, and some open problems.

\section{Preliminaries}\label{SecPrelim}
Recall \cite{R} that the $S_n$-module $Lie_n$ is the action of $S_n$ on the multilinear component of the free Lie algebra, and coincides with the  induced representation $\exp(\frac{2i\pi}{n})\uparrow_{C_n}^{S_n},$ where $C_n$ is the cyclic group generated by an $n$-cycle in $S_n$.  

Another module that will be of interest is the $S_n$-module $C\!onj_n$ afforded by the conjugacy action of $S_n$ on the class of $n$-cycles.  Clearly we have $C\!onj_n\simeq {\mathbf 1}\uparrow_{C_n}^{S_n}.$

We follow \cite{M} and \cite{St4EC2} for notation regarding symmetric functions.  In particular, $h_n,$ $e_n$ and $p_n$ denote respectively the complete homogeneous, elementary and power sum symmetric functions.  If $\mathrm{ch}$ is the Frobenius characteristic map from the representation ring of the symmetric group $S_n$ to the ring of symmetric functions with real coefficients, then 
$h_n=\mathrm{ch}(1_{S_n})$ is the characteristic of the trivial representation, and $e_n=\mathrm{ch}({\rm sgn}_{S_n})$ is the characteristic of the sign representation of $S_n.$   If $\mu$ is a partition of $n$ then 
define $p_\mu =\prod_i p_{\mu_i};$ $h_\mu$ and $e_\mu$ are defined multiplicatively in analogous fashion.  The Schur function $s_\mu$  indexed by the partition $\mu$ is the Frobenius characteristic of the $S_n$-irreducible indexed by $\mu.$  
%%%%%%%
  Finally, $\omega$ is the involution on the ring of symmetric functions which takes $h_n$ to $e_n,$ corresponding to tensoring with the sign representation.
  
  If $q$ and $r$ are characteristics of representations of $S_m$ and $S_n$ respectively, they yield a representation of the wreath product $S_m[S_n]$ in a natural way, with the property that when this representation is induced up to $S_{mn},$ its Frobenius characteristic is the plethysm $q[r].$ For more background about this operation, see \cite{M}.  We will make extensive use of the properties of this operation, in particular the fact that plethysm with a symmetric function $r$ is an endomorphism on the ring of symmetric functions \cite[(8.3)]{M}.  See also \cite[Chapter 7, Appendix 2, A2.6]{St4EC2}.

Define \begin{align}\label{defHE} &H(t):=\sum_{i\geq 0}t^i  h_i=\exp \sum_{i\ge 1}\frac{t^ip_i}{i}, \quad E(t) := \sum_{i\geq 0} t^i  e_i=\exp \sum_{i\ge 1}(-1)^{i-1}\frac{t^ip_i}{i}; \\
&H:=\sum_{i\geq 0}  h_i, \quad E:= \sum_{i\geq 0}  e_i; \quad H^{\pm}:=\sum_{r\geq 0} (-1)^r h_r, \quad E^{\pm}:=\sum_{r\geq 0} (-1)^r e_r. 
\end{align}

Now let $\{q_i\}_{i\geq 1}$ be a sequence of symmetric functions, each $q_i$ homogeneous of degree $i.$    
Let $Q:=\sum_{i\geq 1}q_i,$  $Q(t):=\sum_{n\geq 1} t^n q_n.$
For each partition $\lambda$ of $n\geq 1$  with $m_i(\lambda)=m_i$ parts equal to $i\geq 1,$ let $|\lambda|=n = \sum_{i\geq 1} i m_i$ be the size of $\lambda,$ and 
$\ell(\lambda)=\sum_{i\geq 1} m_i(\lambda)=\sum_{i\geq 1} m_i$ be the length (total number of parts) of $\lambda.$ 
%%%%%

%%%%%

Define 
\begin{equation}\label{HigherQ}
 H_\lambda[Q] := \prod_{i:m_i(\lambda)\geq 1} h_{m_i}[q_i],\qquad  \qquad E_\lambda[Q] :=\prod_{i:m_i(\lambda)\geq 1} e_{m_i}[q_i].
 \end{equation}

For the empty partition (of zero) we define $H_\emptyset [Q]=1=
E_\emptyset[Q]=H^{\pm}_\emptyset [Q]=
E^{\pm}_\emptyset[Q].$

 Consider the generating functions 
$H[Q](t)$ and $E[Q](t).$ 
With the convention that $Par$, the set of all partitions of nonnegative integers,
 includes the unique empty partition of zero,  by the preceding observations and standard properties of plethysm \cite{M} we have 
 
 \begin{align}\label{defhrofQ} &h_r[Q]|_{{\rm deg\ }n}=\sum_{\stackrel {\lambda\vdash n}{\ell(\lambda)=r}} H_\lambda[Q], 
 \qquad \text{and } \quad 
 e_r[Q]|_{{\rm deg\ }n}=\sum_{\stackrel{\lambda\vdash n}{\ell(\lambda)=r}}E_\lambda[Q];\\
&H[Q](t)=\sum_{\lambda\in Par} t^{|\lambda|} H_\lambda[Q], \qquad \text{and } \quad E[Q](t)=\sum_{\lambda\in Par} t^{|\lambda|} E_\lambda[Q].\end{align}

We will also let $D\!Par$ denote the subset of $Par$ consisting of the empty partition and all partitions with distinct parts.
%%%
%%%

Also write $Q^{alt}(t)$ for the alternating sum $\sum_{n\geq 1} (-1)^{i-1} t^i q_i = t q_1-t^2 q_2+t^3 q_3-\ldots.$  

Let $\psi(n)$ be any real-valued function defined on the positive integers. 
Define symmetric functions $f_n$ by 
 \begin{equation}\label{definef_n}f_n := \dfrac{1}{n} \sum_{d|n} \psi(d) p_d^{\frac{n}{d}}, \quad
\text{so that} \quad  \omega(f_n) =  \dfrac{1}{n} \sum_{d|n} \psi(d) (-1)^{n-\frac{n}{d}} p_d^{\frac{n}{d}}.\end{equation}
Note that,  when $\psi(1)$ is a positive integer,  this makes $f_n$ the Frobenius characteristic of a possibly virtual $S_n$-module whose dimension is $(n-1)!\psi(1).$

For each symmetric function $f_n$ of the form~\eqref{definef_n}, define the associated polynomial $\tilde f_n(t)$ in one variable, $t,$ by
\begin{equation}\label{definepolyf_n}
\tilde f_n(t) :=\dfrac{1}{n} \sum_{d|n} \psi(d) t^{\frac{n}{d}}.
\end{equation}

By abuse of notation we will also write $Lie_n$ (resp. $C\!onj_n$) to mean the Frobenius characteristic of the representation $Lie_n$ (resp. $C\!onj_n$).  Let $\mu(d)$ denote the number-theoretic M\"obius function, and $\phi(d)$ the Euler totient function.  The following facts are well known (see \cite{R}). 
\begin{equation}\label{LieConj}Lie_n={\rm \, ch\,} \exp(\tfrac{2i\pi}{n})\!\uparrow_{C_n}^{S_n} = \frac{1}{n}\sum_{d|n} \mu(d) p_d^{\frac{n}{d}};\ 
C\!onj_n={\rm \, ch\,} \mathbf{1}\!\uparrow_{C_n}^{S_n}=\frac{1}{n}\sum_{d|n} \phi(d) p_d^{\frac{n}{d}}
\end{equation}
Finally we define the series
 \begin{equation} \quad Lie:=\sum_{i\geq 1} Lie_i; \quad C\!onj:=\sum_{i\geq 1}C\!onj_i.
\end{equation} 
Note that the equations ~\eqref{HigherQ} give, for $Q=Lie,$   the higher Lie modules $H_\lambda[Lie]$ (as defined by Thrall in \cite{T})  and $E_\lambda[Lie].$

  The important tools used to establish the results of this paper are the theorems of \cite{Su1}, which give uniform formulas for symmetric and exterior powers of modules induced from centralisers.  
Recall the following results regarding the sequence of symmetric functions $f_n$ defined in equation (\ref{definef_n}):
\begin{prop}\label{Su1Prop3.1}\cite[Proposition 3.1]{Su1}  Define $F(t):=\sum_{i\geq 1} t^i f_i,$   and define 
$(\omega F)^{alt}(t):=\sum_{i\geq 1} (-1)^{i-1} t^i\omega(f_i).  $  Then 
\begin{align} &F(t)= \log \prod_{d\geq 1} (1-t^d p_d)^{-\frac{\psi(d)}{d}}\\ 
&(\omega F)^{alt}(t)= \log \prod_{d\geq 1} (1+t^d p_d)^{\frac{\psi(d)}{d}}\end{align}
\end{prop}

\begin{thm}\label{metathm} \cite[Theorem 3.2]{Su1} Let $F=\sum_{n\geq 1}  f_n$ where the symmetric function $f_n$ is of the form (\ref{definef_n}), $H(v)=\sum_{n\geq 0} v^n h_n$ and 
$E(v)=\sum_{n\geq 0}  v^n e_n.$  
We have the following plethystic generating functions:
\begin{equation}\text{(Symmetric powers)  }\label{metaSym}H(v)[F] = \sum_{\lambda\in Par} v^{\ell(\lambda)}  H_\lambda[F]=\prod_{m\geq 1} (1-p_m)^ {-\tilde f_m(v)}\end{equation}
\begin{equation}\text{(Exterior powers)  }\label{metaExt}E(v)[F]=\sum_{\lambda\in Par} v^{\ell(\lambda)} E_\lambda[F] =\prod_{m\geq 1} (1- p_m)^{\tilde f_m(-v)}\end{equation}
\begin{center}$\sum_{\lambda\in Par}  (-1)^{|\lambda|-\ell(\lambda)}v^{\ell(\lambda)} \omega(E_\lambda[F])$ \text{ (Alternating exterior powers)  }\end{center}
\begin{equation}\label{metaAltExt}=\sum_{\lambda \in Par} v^{\ell(\lambda)}H_\lambda[\omega(F)^{alt}]
=H(v)[\omega(F)^{alt}]
= \prod_{m\geq 1} (1+ p_m)^{\tilde f_m(v)}\end{equation}
\begin{center}$\sum_{\lambda\in Par} (-1)^{|\lambda|-\ell(\lambda)}v^{\ell(\lambda)} \omega(H_\lambda[F])$ \text{ (Alternating symmetric powers)  }\end{center}
\begin{equation}\label{metaAltSym}=\sum_{\lambda \in Par} v^{\ell(\lambda)}E_\lambda[\omega(F)^{alt}]
=E(v)[\omega(F)^{alt}]
= \prod_{m\geq 1} (1+ p_m)^{-\tilde f_m(-v)}.\end{equation}
\end{thm}

\begin{prop}\label{metaf} \cite[Lemma 3.3]{Su1} The numbers $\tilde f_n(1)$ and $\tilde f_n(-1)$ determine each other according to the equations 
$\tilde f_{2m+1}(-1)=-\tilde f_{2m+1}(1)$ for all $m\geq 0,$ and 
$\tilde f_{2m}(-1)=\tilde f_m(1) -\tilde f_{2m}(1)$ for all $m\geq 1.$  In fact, 
the symmetric functions 
$f_n = \dfrac{1}{n} \sum_{d|n} \psi(d) p_d^{\frac{n}{d}}$
are  determined by the numbers 
 $\tilde f_n(1)=\dfrac{1}{n} \sum_{d|n}\psi(d) ,$
or by the numbers $\tilde f_n(-1)=\dfrac{1}{n} \sum_{d|n}\psi(d) (-1)^{\frac{n}{d}}.$
\end{prop}

\begin{lem}\label{pm}({\cite{SuLieSup2}}) Let $F=\sum_{n\geq 1} f_n,$  $G=1+\sum_{n\geq 1} g_n$ and $K=1+\sum_{n\geq 1} k_n$ be arbitrary formal series of symmetric functions, as usual with $f_n, g_n, k_n$ being of homogeneous degree $n.$
\begin{enumerate}
\item 
$H[F]=G\iff E^{\pm}[F]=\dfrac{1}{G} \iff \sum_{r\geq 1} (-1)^{r-1} e_r[F] =\dfrac{G-1}{G}.$
\item $E[F]=K\iff H^{\pm}[F]=\dfrac{1}{K} \iff \sum_{r\geq 1} (-1)^{r-1} h_r[F] =\dfrac{K-1}{K}.$
\end{enumerate}
\end{lem}

\begin{prop}\label{metaequiv}(\cite{SuLieSup2}) The identities of Theorem \ref{metathm} are all equivalent to 
\begin{align}E^{\pm}(v)[F] 
&=\prod_{m\geq 1} (1- p_m)^ {\tilde f_m(v)}, \text{ and}\\
%\end{equation}
%\item 
%
%\begin{equation}
H^{\pm}(v)[F]&=\prod_{m\geq 1} (1- p_m)^{-\tilde f_m(-v)}%\end{equation}
\end{align}
\end{prop}

\section{A class of symmetric functions indexed by subsets of primes}\label{SecPrimes}

In this section we will prove some of the key Schur positivity results on sums of power sums determined by sets of primes.

 We begin by stating a theorem of Foulkes  on the character values of representations induced from the cyclic subgroup $C_n$ of $S_n$ generated by the long cycle of length $n,$  which 
asserts Part (1) of the following (see also \cite[Ex. 7.88]{St4EC2}). We refer the reader to \cite{St4EC2} for the definition of the major index statistic on tableaux. 
\begin{thm}\label{Foulkes}   Let $\ell_n^{(r)}$ denote the Frobenius characteristic of the induced representation $\exp\left(\frac{2i\pi}{n}\cdot r\right)\big\uparrow_{C_n}^{S_n},$  $1\leq r  \leq n.$  Let $\psi_r(d)$ denote the expression 
$\phi(d) \dfrac{\mu(\frac{d}{(d, r)})} {\phi(\frac{d}{(d, r)})}.$ Then 
\begin{enumerate}
\item (Foulkes) \cite{F}
$\ell_n^{(r)}=\dfrac{1}{n} \sum_{d|n} \psi_r(d) p_d^\frac{n}{d}.$
\item (\cite[(7.191)]{KW},  \cite{St4EC2}) The multiplicity of the Schur function $s_\lambda$ in the Schur function expansion of $\ell_n^{(r)}$ is the number of standard Young tableaux of shape $\lambda$ with major index congruent to $r$ modulo $n.$
\end{enumerate}
\end{thm}

The quantity $\psi_r(d)$ in Foulkes' theorem is a \textit{Ramanujan sum}; it equals the sum of the $r$th powers of all the primitive $d$th roots of unity.

We are now ready to generalise the definitions ~\eqref{LieConj}  of $Lie_n$ and $C\!onj_n.$

\begin{defn}\label{SetofPrimes} Let $S=\{q_1, \ldots, q_k,\ldots\}$ be a set of distinct primes.  Every positive integer $n$ factors uniquely into $n=Q_n\ell_n$ where $Q_n= \prod_{q\in S} q^{a_q(n)} $ for nonnegative integers $a_q(n),$ and $(\ell_n, q)=1$ for all $q\in S $ such that $a_q(n)\geq 1.$ We associate to the set $S$ a representation $Lie_n^S$, defined via its Frobenius characteristic, as follows:
\begin{equation*}\label{SetofPrimes_a}Lie_n^S= \mathrm{ch}\exp \left(\tfrac{2\pi i}{n}\cdot Q_n\right){\big\uparrow}_{C_n}^{S_n}
\end{equation*} 
If $\bar{S}$ denotes the set of primes \textit{not} in $S,$ then clearly we have
\begin{equation*}\label{SetofPrimes_b}
Lie_n^{\bar{S}}= \mathrm{ch }\exp \left(\tfrac{2\pi i}{n}\cdot \ell_n\right)\big\uparrow_{C_n}^{S_n}.\end{equation*}
We may allow $S$ to be the empty set in this definition by interpreting $Q_n$ as the empty product; we then have $Q_n=1$ for all $n\geq 1,$ and hence 
\begin{equation*} Lie_n^\emptyset=Lie_n=\ell_n^{(1)},\quad {\rm  while }\quad
Lie_n^{\bar{\emptyset}}=C\!onj_n=\ell_n^{(n)};
\end{equation*} thus  $Lie_n^{\bar{\emptyset}}$ is the Frobenius characteristic of $S_n$ acting on the class of $n$-cycles by conjugation. Similarly, at the other extreme, if $S$ is the set of all primes $\mathcal{P}$, then $Q_n=n, \ell_n=1$ for all $n,$ and thus 
\begin{equation*} 
Lie_n^{\mathcal{P}}=C\!onj_n=\ell_n^{(n)} \quad {\rm  while }\quad Lie_n^{\bar {\mathcal{P}} }=Lie_n=\ell_n^{(1)}.
\end{equation*}
Let $P(S)$ denote the set of positive integers whose  prime divisors constitute a subset of the set of primes $S;$  note that $1\in P(S).$ Similarly let $P(\bar{S})$ be the set of positive integers whose set of prime divisors is disjoint from  $S;$ note that $P(\bar{S})$ is precisely the set of integers that are relatively prime to every prime in $S.$   Also $P(S)\cap P(\bar{S})=\{1\}.$  Equivalently, 
\begin{equation*}P(S)=\{n\geq 1: q \text{ is a prime factor of }n \Longrightarrow q\in S \},\end{equation*}
\begin{equation*}P(\bar{S})=\{n\geq 1: q \text{ is a prime factor of }n \Longrightarrow q\notin S\}.\end{equation*}
\end{defn}

When $S$ consists of a single prime $\{q\}$, we can view the modules $Lie_n^{(q)}$ as interpolating between the modules $Lie_n$ and $C\!onj_n=\text{ch } 1\uparrow_{C_n}^{S_n},$ since:
\begin{equation}
 Lie_n^{(q)}=Lie_n \text{ when $n$ is relatively prime to the prime $q,$}  
\end{equation}
and, at the other extreme:
\begin{equation}
Lie_n^{(q)}=C\!onj_n
 \text{ when $n$ is a power of the prime $q.$ }
\end{equation}

More generally, if $S$ is a nonempty set of primes, then 
\begin{equation} Lie_n^S=
\begin{cases} Lie_n, & n\in P(\bar{S})\\
                  C\!onj_n, &n\in P(S).
\end{cases}
\end{equation}
 
Our goal is to describe the symmetric and exterior powers of these 
modules, that is, the analogues of the higher Lie modules $H_\lambda[Lie]$  and $E_\lambda[Lie].$  Theorem \ref{metathm} of Section~\ref{SecPrelim} allows us to do this in an elegant manner, without intricate plethystic calculations. In order to apply the theorem, we must first determine the values $\widetilde{Lie}_n^S(\pm 1)$ of the polynomial $\widetilde{Lie}_n^S(t)$ associated to the symmetric function $Lie_n^S$ (see~\eqref{definepolyf_n}), for each subset $S$ of primes.  For this we need a computation involving the Ramanujan sum appearing in Theorem~\ref{Foulkes}.

\begin{lem}\label{Ramanujan} Let $r$ be a divisor of $n$ such that $(r, \frac{n}{r})=1.$ Consider the representation with Frobenius characteristic 
$\ell_n^{(r)} =\mathrm{ch}\exp (\tfrac{2\pi i}{n}\cdot r)\big\uparrow_{C_n}^{S_n}.$
Then for the associated polynomial $\tilde \ell_n^{(r)}(t),$ we have $\tilde \ell_n^{(r)}(1)=$
$\begin{cases} 1, & n=r,\\
                    0, &\text{ otherwise.}
\end{cases}$
\end{lem}
\begin{proof} For ease of notation we write $n=rs$ where $r,s$ are relatively prime positive integers.  Then every divisor $d$ of $n$ factors uniquely into $d=r_1 s_1$ where $r_1$ is a divisor of $r$, 
$s_1$ is a divisor of $s,$ and $r_1,$ $s_1$ are thus relatively prime. 
In particular $(d,r)=r_1,$ and $\frac{d}{(d,r)}=s_1.$  Hence we have 
\begin{align*} \tilde\ell_n^{(r)}(1)&=\frac{1}{n}\sum_{d|n}\phi(d) \frac {\mu(\frac{d}{(d,r)})}{\phi(\frac{d}{(d, r)})}
=\frac{1}{n}\sum_{{r_1|r},{s_1|s}} \phi(r_1 s_1) 
\frac{\mu(\frac{d}{r_1})}{\phi(\frac{d}{r_1})}\\
&=\frac{1}{n}\sum_{{r_1|r},{s_1|s}} \phi(r_1 s_1) 
\frac{\mu(s_1)}{\phi(s_1)}
=\frac{1}{n}\sum_{{r_1|r},{s_1|s}} \phi(r_1) 
\mu(s_1)\\
&=\frac{1}{n}\left(\sum_{r_1|r} \phi(r_1)\right )\left(\sum_{s_1|s}\mu(s_1)\right)=\frac{r}{n}\sum_{s_1|s}\mu(s_1).
\end{align*}
The last sum is nonzero if and only if $s=\frac{n}{r}=1,$ i.e. $n=r,$ and the result follows.
\end{proof}

\begin{cor}\label{LSvalues} Let $S$ be a fixed  set of primes, and let $Lie_n^S$ 
be (the Frobenius characteristic of) the representation of Definition \ref{SetofPrimes}.  Then 
\begin{enumerate} \item $\widetilde{Lie}_n^S(1)$ is nonzero if and only if $n=1$ or the distinct prime factors of $n$ are contained in the set $S,$ i.e. if and only if $n\in P(S),$ in which case it equals 1. 
\item In particular $\widetilde{Lie}_n^\emptyset(1)=\widetilde{Lie}_n(1)=\widetilde{Lie}_n^{\bar{\mathcal{P}}}(1)$ is nonzero unless $n=1,$ in which case it equals 1.  Similarly $\widetilde{Lie}_n^{\bar{\emptyset}}(1)=\widetilde{C\!onj}_n(1)$ equals 1 for all $n\geq 1.$
\end{enumerate}
\end{cor}
\begin{proof} Write $n=Q_n\ell_n$ where $Q_n$ is the product of all maximal prime powers $q_i^{a_i}$ that are divisors of $n,$ and $\ell_n$ is relatively prime to every prime $q_i$ in $S.$   In Lemma \ref{Ramanujan}, take $r=Q_n$ for Part (1). 
Lemma \ref{Ramanujan} says $\widetilde{Lie}_n^S(1)$ 
 if and only if $Q_n=n$ 
 Note that this applies also when $S$ is empty 
 since then $Q_n=1.$ 
 The claims follow immediately from the observations succeeding Definition \ref{SetofPrimes}.
\end{proof}

The next step is to apply  Theorem \ref{metathm} to the sequence of symmetric functions $f_n=Lie_n^S.$  
 In order to do this, we must verify that each sequence is determined by a single function $\psi$ whose value depends only on the set $S.$   (See equation~(\ref{definef_n}).) 
Using Theorem \ref{Foulkes}, we see that with $Q_n, \ell_n$ as in Definition \ref{SetofPrimes}, one has 
\begin{center}$Lie_n^S=\frac{1}{n}\sum_{d|n} \psi_n(d) p_d^{\frac{n}{d}}$\quad
with \quad
$\psi_n(d)=\phi(d) \dfrac{\mu(d/(d,Q_n))}{\phi(d/(d,Q_n))}.$\end{center}
We claim that this expression is in fact independent of $n.$ 
Since $d$ divides $n,$ we can factor $d$ uniquely as $d=Q_d\ell_d$ where $\ell_d$ is relatively prime to $Q_d$ as well as to all the primes in the set $S.$ 
In particular $(d, Q_n)=Q_d.$ 
Using the multiplicative property of $\phi,$ we obtain
$$\psi_n(d)=\phi(Q_d)\phi(\ell_d) \dfrac{\mu(d/Q_d)}{\phi(d/(Q_d))}
=\phi(Q_d)\phi(\ell_d) \dfrac{\mu(\ell_d)}{\phi(\ell_d)}
=\phi(Q_d)\mu(\ell_d),$$
thereby showing that $\psi_n(d)=\psi(d)$ depends only on $d$ and the set of primes $S.$ 
An entirely analogous calculation shows that 
\begin{center}$Lie_n^{\bar S}=\frac{1}{n}\sum_{d|n} {\bar \psi}(d) p_d^{\frac{n}{d}}$\quad
with\quad  $\bar\psi(d)=\phi(\ell_d)\mu(Q_d).$\end{center} 
In summary, given a set of primes $S$ and a positive integer $d,$ if $d$ is factored uniquely as $d=Q_d\ell_d$ where $\ell_d$ is relatively prime to the primes in $S$ and also to $Q_d,$ (so that $Q_d$ is a product of prime powers of elements of $S$), then 
\begin{equation}\label{FoulkesPrimes1} \text{For all }n \geq 1,\quad
Lie_n^S=\frac{1}{n}\sum_{d|n} \psi(d) p_d^{\frac{n}{d}} \quad 
\text{ with } \psi(d)=\phi(Q_d)\mu(\ell_d), \text{ and}
\end{equation}
\begin{equation}\label{FoulkesPrimes2} \text{For all }n \geq 1,\quad
Lie_n^{\bar S}=\frac{1}{n}\sum_{d|n} \bar\psi(d) p_d^{\frac{n}{d}} \quad
\text{ with } \bar\psi(d)=\phi(\ell_d)\mu(Q_d).
\end{equation}

With the calculation of $\widetilde{Lie}_n^S(1)$ from Corollary~\ref{LSvalues}, we are ready to invoke the power of Theorem~\ref{metathm}.
>From equations \eqref{metaSym} and \eqref{metaAltExt}, we have the following formula for the higher $L^S$-modules:
\begin{thm}\label{SymAltSymLS} Let $L^S=\sum_{n\geq 1} Lie_n^S. $   
 Then one has the generating functions
   \begin{equation}\label{SymLS}
 H[L^S](t) = \prod_{n\in P(S)}  (1-t^n p_n)^{-1}
=\sum_{\lambda\in Par:\lambda_i\in P(S)} t^{|\lambda|} p_\lambda;
\end{equation}
\begin{equation}\label{AltSymLS} H[\omega(L^S)^{alt}](t)=\prod_{n\in P(S)} (1+t^n p_n)
=\sum_{\lambda\in D\!Par:\lambda_i\in P(S)} t^{|\lambda|} p_\lambda;
\end{equation} \end{thm}

  Our first Schur positivity result is now easily deduced. Recalling the definitions of $P(S), P(\bar{S})$ from Definition \ref{SetofPrimes}, we obtain:
\begin{thm}\label{Spositivity1} Fix a set $S$ of primes.  The following sums of power sums are Schur positive:
\begin{enumerate}
\item $\sum_{\lambda\vdash n: \lambda_i\in P(S)} p_\lambda;$
\item $\sum_{\lambda\vdash n: \lambda_i\in P(\bar{S})} p_\lambda;$
\item $\sum_{\stackrel{\lambda\vdash n: \lambda_i\in P({S})}{n-\ell(\lambda) \text{ even}}} p_\lambda=\sum_{\stackrel{\lambda\vdash n: \lambda_i\in P({S})}{\lambda 
 \text{ has an even number of even parts}}} p_\lambda.$
\end{enumerate}
\end{thm}
\begin{proof} Parts (1) and (2) are immediate since $Lie_n^S$ and $Lie_n^{\bar{S}}$ are Schur positive by definition, so the left side of (\ref{SymLS}) in each case is a symmetric power of  a true $S_n$-module.
For Part (3), observe first that  $n-\ell(\lambda)$ is congruent to the number of even parts of $\lambda$, and hence Part (3) coincides with Part (1) unless $2\in S.$ In that case, applying the involution $\omega$ to 
equation (\ref{SymLS}) of Theorem \ref{SymAltSymLS}, (with $t=1$), we have 
\begin{equation*} 
\sum_{\lambda\vdash n: \lambda_i\in P({S}), n-\ell(\lambda) \text{ even}} p_\lambda
=\frac{1}{2} (H[L^S]+\omega(H[L^S]))
\end{equation*}
But the Schur expansion on the left-hand side has integer coefficients, and the Schur expansion on the right-hand side is certainly positive.  
Part (3) follows.\end{proof}

In the special case when  $S$ is the set of all odd primes,  equation (\ref{SymLS})  and Part (2) of Theorem \ref{Spositivity1}   describe the sum of power sums $p_\lambda$ for all parts of $\lambda$ odd, as the symmetrised module of a representation, whereas  \cite[Theorem 4.9]{Su1} (see also Proposition \ref{ExtAltExtLS} below) gives a description as  the exterior power of the conjugacy action.

When $S=\emptyset, $   equations (\ref{SymLS}) and (\ref{AltSymLS}) applied to $S$ and $\bar{S}$ reduce to  known formulas, listed below,  of Thrall \cite{T}, Cadogan \cite{C}, and Solomon \cite{So}, respectively.  %Write $\pi_n=\omega(Lie_n).$ 
See also \cite[Ex. 7.71, Ex. 7.88, Ex. 7.89]{St4EC2}
\begin{thm} \label{ThrallPBWCadoganSolomon}
\begin{equation*}\label{ThrallPBW}  (Thrall) \qquad H[\sum_{n\geq 1} Lie_n](t)=(1-tp_1)^{-1}
\end{equation*}
 \begin{equation*}\label{Cadogan} (Cadogan)\qquad  H[\sum_{n\geq 1} (-1)^{n-1} \omega(Lie_n)](t)
=1+tp_1. \end{equation*} 
\begin{equation*}\label{Solomon} (Solomon )\qquad H[\sum_{n\geq 1}  \mathrm{ch}
(1\big\uparrow_{C_n}^{S_n})](t)=\prod_{n\geq 1} (1-t^np_n)^{-1}
\end{equation*}
\end{thm}

Next we compute the values of the polynomial $\widetilde{Lie}_n^S(t)$ for $t=-1.$  Thanks to Proposition~\ref{metaf}, we can avoid another cumbersome Ramanujan sum computation.
\begin{lem}\label{LSnegvalues}  For the set of distinct primes 
$S:$
\begin{enumerate}
\item If $2\notin S,$ then $\widetilde{Lie}_n^S(-1)=
\begin{cases} -1, & n\in P(S)\\
                       1, &n\text{ even and }\frac{n}{2}\in P(S)\\
                       0, &\text{ otherwise.}
\end{cases}$
\item If $2\in S,$ then $\widetilde{Lie}_n^S(-1)=
\begin{cases} -1, &n \text{    odd and } n\in P(S)\\
                       0, &\text{ otherwise.}
\end{cases}$
\item  $\widetilde{Lie}_n^\emptyset(-1)=\widetilde{Lie}_n^{\bar{\mathcal{P}}}(-1)
=\begin{cases} -1, & n=1\\
                       1, &n=2\\
                       0, &\text{ otherwise.}
\end{cases}$
\item    $\widetilde{Lie}_n^{\bar{\emptyset}}(-1)=\widetilde{Lie}_n^{\mathcal{P}}(-1)=
\begin{cases} -1, & n\text{    odd }\\
                     0, &\text{ otherwise.}
\end{cases}$
\end{enumerate}
\end{lem}
\begin{proof} It suffices to address Parts (1) and (2).  
If $n$ is odd, by Proposition \ref{metaf} and Corollary \ref{LSvalues}, we have 

$\widetilde{Lie}_n^S(-1)=-\widetilde{Lie}_n^S(1)=\begin{cases} -1, &n\in P(S)\\ 0, &otherwise. \end{cases}$

If $n$ is even, by Proposition \ref{metaf}, we have
 $\widetilde{Lie}_n^S(-1)=\widetilde{Lie}_{\frac{n}{2}}^S(1)-\widetilde{Lie}_n^S(1).$  Invoking Part (1) of Corollary \ref{LSvalues}, we see that, when $2\in S,$  $\widetilde{Lie}_n^S(1)=1$ if and only if $\widetilde{Lie}_{\frac{n}{2}}^S(1)=1.$  Hence when $2\in S$ and $n$ is even, $\widetilde{Lie}_n^S(-1)=0.$

Now assume $n$ is even and $2\notin S.$ Then $n\notin P(S)$ so $\widetilde{Lie}_n^S(1)=0.$ In this case $\widetilde{Lie}_{\frac{n}{2}}^S(1)\neq 0$ if and only if $\frac{n}{2}\in P(S),$ in which case $\widetilde{Lie}_n^S(-1)=\widetilde{Lie}_{\frac{n}{2}}^S(1)=1.$ 

This establishes Parts (1) and (2), and hence the remaining parts.    
\end{proof}
As an immediate consequence, invoking equations (\ref{metaExt}) and (\ref{metaAltSym}) of Theorem \ref{metathm}, we obtain the exterior power analogue of Theorem \ref{SymAltSymLS}.  
\begin{prop} \label{ExtAltExtLS}Let $S$ be a set of  primes, and let $Lie_n^S$ 
be the symmetric function defined in Theorem \ref{SymAltSymLS}.  Then 
 \begin{equation}\label{ExtLS} E[Lie^S](t)=
\begin{cases} \prod_{ n\in P(S)} (1-t^n p_n)^{-1} \prod_{n\, even, \frac{n}{2}\in P(S)} (1-t^np_n), & 2\notin S\\
\prod_{n\, odd, \, n\in P(S) } (1-t^np_n)^{-1}=\prod_{n\in P(S\backslash\{2\}) } (1-t^np_n)^{-1}, & 2\in S.
\end{cases}
\end{equation}
(Note $n$ is necessarily odd in the first product of the case $2\notin S.$)
\begin{equation}\label{omegaExtLS} \omega(E[Lie^S])(t)=
 \prod_{ n\in P(S)} (1-t^n p_n)^{-1}\prod_{n\, even, \frac{n}{2}\in P(S)} (1+t^np_n),  \text{ provided } 2\notin S
\end{equation}
\begin{equation}\label{AltExtLS} E[\omega(Lie^S)^{alt}](t)=
\begin{cases} \prod_{ n\in P(S)} (1+t^n p_n) \prod_{n\, even, \frac{n}{2}\in P(S)} (1+t^np_n)^{-1}, & 2\notin S\\
\prod_{n\, odd, \, n\in P(S) } (1+t^np_n), & 2\in S.
\end{cases}
\end{equation}
\end{prop}
\begin{proof} The second equation is clearly a consequence of applying the involution $\omega$ to the first, which is obtained directly from equation (\ref{metaExt}) of Theorem \ref{metathm}.  Similarly the third equation  follows directly from (\ref{metaAltSym}).  
\end{proof}

As in Theorem \ref{SymAltSymLS}, when $S=\emptyset,$ the first and third equations above allow us to  recover the formulas of \cite[Corollary 5.2]{Su1} and \cite[Theorem 4.2]{Su1}, which we restate using the equivalent formulations of the second and the third equations, in order to emphasise the fact that the right-hand side may be written as a nonnegative linear combination of power sums.
Recall that we write 
$C\!onj_n=Lie_n^{\mathcal{P}}=\mathrm{ch} (1\big\uparrow_{C_n}^{S_n}).$ 
\begin{prop}\label{ExtLieConj} \cite[Theorem 4.2 and Corollary 5.2]{Su1} We have the generating functions:
\begin{enumerate}[itemsep=8pt] 
\item $\omega(E[Lie](t))=(1+t^2 p_2)(1-tp_1)^{-1}$
\item $E[\sum_{n\geq 1} C\!onj_n](t)=
\prod_{n\geq 1,\, n\, odd} (1-t^np_n)^{-1}$
\item 
$\sum_{\lambda\in Par} (-1)^{|\lambda|-\ell(\lambda)} H_\lambda[Lie] (t)=\omega(E[\omega(Lie)^{alt}])(t)=(1+tp_1)(1-t^2p_2)^{-1}$
\item 
$E[\sum_{n\geq 1} (-1)^{n-1} \omega(C\!onj_n)](t)=
\prod_{n\geq 1,\, n\, odd} (1+t^np_n)$
\end{enumerate}
\end{prop}

For any subset $T_n$ of partitions of $n,$ denote by $P_{T_n}$ the sum of power-sum symmetric functions $\sum_{\lambda\in T_n} p_\lambda.$
Note the passage from symmetric to exterior powers, for the same representation, in Part (1) below.
\begin{thm}\label{Spositivity2} For a fixed set of primes $S,$ the  sums $P_{T_n}$ are Schur positive 
for the following choices of $T_n$:
\begin{enumerate}
\item If $2\in S,$ $T_n=\{\lambda\vdash n: \lambda_i\, odd,\, \lambda_i\in P(S)\};$ in this case the representation coincides with 
$$H[L^{S\backslash\{2\}}]=E[L^S]=\sum_{\lambda:\lambda_i \in P(S\backslash\{2\})} p_\lambda.$$
\item 
If $2\notin S,$ $T_n$ consists of all partitions $\lambda$ of $n$ such that the parts are (necessarily odd and) in $P(S),$ or the parts are twice an odd number in $P(S)$, the even parts occurring at most once.
\end{enumerate}
\end{thm}
\begin{proof}  Part (1) has also been observed in Theorem \ref{Spositivity1}, (2), taking the set of primes to be $S\backslash\{2\}$.  Here it follows from equation (\ref{ExtLS}) of Proposition \ref{ExtAltExtLS}.  In particular we have two ways to realise the representation 
$$\sum_{\lambda:\lambda_i \in P(S\backslash\{2\})} p_\lambda,$$
namely as the symmetric power $H[L^{S\backslash\{2\}}]$ (Theorem \ref{SymAltSymLS}) and also as the exterior power $E[L^S],$ the latter from Proposition \ref{ExtAltExtLS}.

Part (2) follows by extracting the degree $n$ term from the third equation in Proposition \ref{ExtAltExtLS}, noting that we now have exterior powers of Schur positive functions.
\end{proof}

In particular when the set $S$ consists of a single prime $q,$ we have:

\begin{thm}\label{OnePrime} Let $q$ be a fixed prime.  The following  multiplicity-free sums of power sums are the Frobenius characteristics of true representions of $S_n,$ of dimension $n!$
\begin{enumerate}
\item
$\sum_{{\lambda\vdash n}
\atop{ \text{all parts of } \lambda\text{\ are  powers of\ } q} }
p_\lambda \quad \text{(If $q$ is odd,  this representation is self-conjugate.)}$
\item 
$\sum_{\substack{\lambda\vdash n\\(\lambda_i, q)=1\text{ for all }i} }p_\lambda;$ 
\item
$\sum_{\substack{\lambda\vdash n,\lambda_i \text{ odd }\\(\lambda_i, q)=1\text{ for all }i}} p_\lambda, \qquad q \text{ odd}.$

\end{enumerate}
\end{thm}

With $\ell_n^{(k)}$ denoting the characteristic ${\rm ch}\, (\exp(\frac{2\pi i k}{n})\uparrow_{C_n}^{S_n})$ of the Foulkes character, we have the following contrasting result:
\begin{thm}\label{Su1Thm5.6}\cite[Lemma 5.5, Theorem 5.6]{Su1} If $k$ is any fixed positive integer, then $\ell_n^{(k)}(1)$ is nonzero if and only if $n$ divides $k,$ in which case it equals 1.  Hence one has the Schur positive sum 
\begin{equation*}\sum_{\lambda\vdash n: \lambda_i|k} p_\lambda
=H[\sum_{m\geq 1} \ell_m^{(k)} ]\vert_{{\rm\, deg\,} n}.
\end{equation*}
\end{thm}

The  important special case of the single prime 2 was mentioned in the Introduction, and is discussed in detail in   {\cite{SuLieSup2}}. 
Recall from Theorem~\ref{ThrallPBW} the formula of Thrall  decomposing  the regular representation into a sum of higher Lie modules, and the formula of Cadogan giving the plethystic inverse of the homogeneous symmetric functions $H$. Similarly, the higher \textit{exterior} modules for  $Lie_n^{(2)}$  give a decomposition of the regular representation, and describe the plethystic inverse of the elementary symmetric functions $E$.  Here we record these facts:

\begin{thm}\label{PlInv-E}({\cite{SuLieSup2}}) Let $Lie^{(2)}_n$ be the Frobenius characteristic of the induced representation $\exp(\frac{2i\pi}{n}\cdot 2^k)\large\uparrow_{C_n}^{S_n},$  where $k$ is the largest power of 2 which divides $n.$ Then we have the following generating functions:
\begin{enumerate}[itemsep=8pt]
\item \label{ExtLieSup2Reg}\qquad\qquad\qquad\qquad\qquad
$ E\left[\sum_{n\geq 1}  Lie^{(2)}_n\right ](t) = (1-tp_1)^{-1}.$

\text{Equivalently,}
\qquad\quad
$H^{\pm}\left[\sum_{n\geq 1} Lie^{(2)}_n\right](t)=1-tp_1.$

\item\label{PlInv-ELieSup2}$E\left[\sum_{n\geq 1} (-1)^{n-1}\omega(Lie^{(2)}_n)\right](t)=1+tp_1.$

\end{enumerate}
\end{thm}

\section{A formula for $\prod_{n\in T} (1-p_n)^{-1}$}\label{SecfnT}

In this section we explore, for a fixed subset $T$ of positive integers,   the sum of power sums resulting from the product 
$\prod_{n\in T} (1-p_n)^{-1}.$
(This product is 1 if $T$ is the empty set.)  We will further generalise the results of the previous section, introducing a new class of symmetric functions indexed by $T.$  These results were announced in \cite{SuFPSAC2019}.

\begin{defn}\label{def6.1} Fix a nonempty subset $T$ of the positive integers.  For each positive integer $d,$ define a function $\psi^T$ by $\psi^T(d):=\sum_{m|d,\, m\in T} m\,\mu\left(\frac{d}{m}\right).$
\end{defn}
\begin{defn}\label{def6.2} For each nonempty subset $T$ of positive integers, 
 define a sequence of (possibly virtual) representations  indexed by the subset $T,$ with Frobenius characteristic 
$$
f_n^T:=\dfrac{1}{n}\sum_{d|n} \psi^T(d) p_d^{\frac{n}{d}},$$
as well as associated polynomials $\tilde f_n^T(t):=\frac{1}{n}\sum_{d|n} \psi^T(d) t^{\frac{n}{d}}$
 as in~\eqref{definepolyf_n}. Define the symmetric function series $F^T:=\sum_{n\geq 1} f_n^T.$  Finally let 
 $p^T:=\sum_{n\in T} p_n.$
\end{defn}

These definitions imply the following facts, whose relevance will become clear in the remarks following Corollary~\ref{cor6.5}.
\begin{lem}\label{lem6.3} $\tilde f_n^T(1)=1$ if and only if $n\in T,$ and $\tilde f_n^T(1)=0$ otherwise.
\end{lem}
\begin{proof} Let us write $\delta(m\in T)$ for the indicator function of the set $T,$ so that $\delta(m\in T)=1$ if and only if $m\in T,$ and is zero otherwise.

By defnition of $\psi^T$, we have 
$\psi^T(n)=\sum_{d|n}\mu(\tfrac{n}{d})\ d\,\delta(d\in T).$
Hence M\"obius inversion gives 
$n\delta(n\in T)=\sum_{d|n}\psi^T(d)
=n\, \tilde f_n^T(1),$
i.e. $\tilde f_n^T(1)=\delta(n\in T)$ as claimed.
\end{proof}

With this lemma, we can now determine a formula for the \lq\lq higher $F^T$ modules":
\begin{thm}\label{thm6.4} Let $T$ be a nonempty subset of the positive integers.  Then:
\begin{equation}\label{thm6.4eqn1} H[F^T]=\prod_{n\in T} (1-p_n)^{-1}
\end{equation}
\begin{equation}\label{thm6.4eqn2} F^T=p^T[Lie]=\sum_{m\in T} Lie[p_m],  \ \mathrm{or\ equivalently\ } 
f_n^T=\sum_{\stackrel{m\in T}{m|n}} Lie_{\frac{n}{m}}[p_m].
\end{equation}
If $G^T=\sum_{k\geq 0}\sum_{m\in T} Lie[p_{m\cdot 2^k}],$ then 
\begin{equation}\label{thm6.4eqn3}
E[G^T]=\prod_{n\in T} (1-p_n)^{-1}=H[F^T].
\end{equation}
\end{thm}

\begin{proof} 
Equation (\ref{thm6.4eqn1}) is immediate from Theorem \ref{metathm} and Lemma \ref{lem6.3}. 
 From Proposition \ref{Su1Prop3.1}, we have 
\begin{align*}\exp (F^T)&= \prod_{d\geq 1} (1-p_d)^{-\frac{1}{d}\psi^T(d)} =\prod_{d\geq 1} (1-p_d)^{-\sum_{m|d,\, m\in T} \frac{m}{d}\mu(\frac{d}{m})} \\
&=\prod_{d\geq 1} \prod_{m|d,\, m\in T} (1-p_d)^{- \frac{m}{d}\mu(\frac{d}{m})}
=\prod_{ m\in T}\prod_{r\geq 1}(1-p_{rm})^{- \frac{1}{r}\mu(r)}, \mathrm{\ putting\ }d=rm\\
&=\prod_{ m\in T}\prod_{r\geq 1}(1-p_{r})^{- \frac{1}{r}\mu(r)}[p_m]
=\prod_{ m\in T} \exp(Lie[p_m] )=\exp\sum_{m\in T} Lie[p_m],
\end{align*}
where we have used the first equation in Proposition \ref{Su1Prop3.1} for $F=Lie.$  Equation (\ref{thm6.4eqn2}) follows.
Equation (\ref{thm6.4eqn3})  is a consequence of Proposition \ref{Pleth3}.
\end{proof}

\begin{cor}\label{cor6.5} If either $F^T$ or $G^T$ is Schur positive, then so is 
$$\prod_{n\in T} (1-p_n)^{-1}=\sum_{\stackrel{\lambda\in Par}{\lambda_i\in T}}p_\lambda.$$
\end{cor}

The following remarks will explain the connection between the functions $f_n^T$ and the work of Section 3.  
\begin{enumerate}
\item If $T=\{1\},$ then $\psi^T(d)=\mu(d)$ and $f_n^T$ corresponds to the representation $Lie_n.$
\item If $T$ is the set of all positive integers, then $\psi^T(d)=\phi(d)$ by M\"obius inversion of the well-known identity $m=\sum_{d|m}\phi(d).$  Thus $f_n^T$ is the characteristic of the conjugacy action
$\textbf{ 1}\uparrow_{C_n}^{S_n},$ i.e. $f_n^T=C\!onj_n.$
\item Fix a set $S$ of primes.  Let $T$ be the set of all integers whose prime factors are all in $S.$  Then clearly if $d\in T,$ $\psi^T(d)=\phi(d)$ by the identity used in (2). Otherwise $d=Q_d\ell_d$ with $Q_d\in T$ and $\ell_d$ relatively prime to $Q_d$ and also relatively prime to all integers in $T.$ Hence, since  $\mu$ is multiplicative,
Definition \ref{def6.1} gives 
$$\psi^T(d)=\sum_{m\in T} m \mu(Q_d/m) \mu(\ell_d)
=\mu(\ell_d)\cdot \sum_{m\in T} m \mu(Q_d/m)=\mu(\ell_d)\psi^T(Q_d),$$ 
 Since $Q_d\in T,$  we obtain  $\psi^T(d)=\mu(\ell_d)\phi(Q_d),$ which is precisely the formula given by (\ref{FoulkesPrimes1}). Thus $f_n^T=Lie_n^S.$
\item Fix a set $S$ of primes. Now let $T$ be the set of all integers relatively prime to every element of $S$ (so $T$ is the set of integers none of whose prime factors is in $S$). Exactly as above, we see that Definition \ref{def6.1} reduces to (\ref{FoulkesPrimes2}), and hence $f_n^T$ is the characteristic $Lie_n^{\bar S}.$
\end{enumerate}

>From Theorem \ref{thm6.4} and the preceding observations, we have the following decompositions of the representations $C\!onj_n,$ $Lie_n^{(q)}.$  Only ~\eqref{prop6.6eqn4a}, recorded here for completeness, requires comment; it follows from ~\eqref{prop6.6eqn4} and Proposition~\ref{Pleth1} in the next section.  
\begin{prop}\label{prop6.6}
 \begin{equation}\label{prop6.6eqn1}\sum_{m\geq 1} p_m[Lie]=\sum_{n\geq 1} C\!onj_n;
\end{equation}
\begin{equation}\label{prop6.6eqn2}\sum_{m\geq 1} p_m=\sum_{n\geq 1}C\!onj_n[\sum_{r\geq 1}(-1)^{r-1} e_r]
\end{equation}
The plethystic inverse of $ C\!onj$ is 
\begin{equation}\label{prop6.6eqn3} (\sum_{n\geq 1} C\!onj_n)^{\langle -1\rangle}
= \sum_{r\geq 1}(-1)^{r-1} e_r[\sum_{n\geq 1} \mu(n) p_n].
\end{equation}
Let $q$ be prime, and let $n=\ell q^k $ 
where $(\ell, q)=1.$ Then 
\begin{equation}\label{prop6.6eqn4} Lie_n^{(q)} =\sum_{r=0}^k Lie_{\ell q^{k-r}}[p_{q^r}].\end{equation}
\begin{equation}\label{prop6.6eqn4a} Lie^{(q)} =\sum_{r\ge 0} Lie[p_{q^r}] \text{ and hence } Lie=Lie^{(q)}[p_1-p_q].
\end{equation}
The plethystic inverse of $Lie^{(q)}$ is 
\begin{equation}\label{prop6.6eqn5}\sum_{n\geq 1}(Lie_n^{(q)})^{\langle -1\rangle}
=Lie^{\langle -1\rangle}[p_1-p_q]=(\sum_{r\geq 1} (-1)^{r-1} e_r)[p_1-p_q]. 
\end{equation}

\end{prop}

The construct of Definition \ref{def6.2} allows us to remove the restriction that $q$ be prime, as follows.  Let $k\geq 2$ be any positive integer, and take $T$ to be the set of all nonnegative powers of $k.$ In this case Theorem \ref{thm6.4} gives
\begin{equation}\label{eqn6.9} H[\sum_{n\geq 1} f_n^T]=\prod_{r\geq 0} (1-p_{k^r})^{-1}, \qquad \sum_{n\geq 1} f_n^T=\sum_{r\geq 0} p_{k^r}[Lie].
\end{equation}
By inverting this equation plethystically, we obtain the recurrence 
\begin{equation}\label{eqn6.10} \text{For } k\geq 2, \qquad f_n^T=\begin{cases}  Lie_n+f_{\frac{n}{k}}[p_k], & k|n;\\
                          Lie_n, &\mathrm{otherwise},\\
\end{cases}
\end{equation}

However computations show that for $k=4,$ $f_n^T$ is not Schur positive when $n=4, 16,$ and the degree 16 term in the product 
$\prod_{r\geq 0} (1-p_{4^r})^{-1}$ is not Schur positive.  In both cases it is the sign representation that appears with coefficient $(-1).$

\begin{conj} For any {odd} positive integer $k,$ $f_n^T$ as defined above is Schur positive.  
\end{conj}

\begin{conj} The product $\prod_{r\geq 0} (1-p_{k^r})^{-1}$ is Schur positive for any {odd} positive integer $k.$
\end{conj}

Fix $k\geq 2$ and consider the subset $T=\{1,k\}.$ It was shown in \cite[Theorem 4.23]{Su1} that the symmetric function 
$$W_{n,k}=\sum_{\mu\vdash n, \mu_i=1\, or \, k} p_\mu$$ is Schur positive.  Define $W_{0,k}=1.$ Then
$$\sum_{n\geq 0} W_{n,k}=\prod_{n\in T}(1-p_n)^{-1}=(1-p_1)^{-1} (1-p_k)^{-1}.$$
For $k=1$ we set $W_{n,1}=p_1^n$ for all $n\geq 0,$ so that the  preceding equation reduces, as expected, to
\begin{center} 
$\sum_{n\geq 0} W_{n,1}=\prod_{n\in T}(1-p_n)^{-1}=(1-p_1)^{-1}.$\end{center}

\begin{prop}\label{prop6.7} If $T=\{1,k\}$ and $k\geq 2,$ then 
\begin{equation}\label{eqn6.11}f_n^T
=\begin{cases}  Lie_n+Lie_{\frac{n}{k}}[p_k], & k|n;\\
                          Lie_n, &\mathrm{otherwise},\\
\end{cases}
\end{equation}
and hence $\sum_{n\geq 0} W_{n,k}=H[\sum_{n\geq 0} f_n^T].$

If $k$ is  prime, then 
$f_n^{\{1,k\}}=\mathrm{ch} (\exp\frac{2ki\pi}{n})\big\uparrow_{C_n}^{S_n}=\ell_n^{(k)},$ 
and hence the symmetric function defined by (\ref{eqn6.11}) is  
Schur positive. \end{prop}
\begin{proof} Equation (\ref{eqn6.11}) is immediate from Theorem \ref{thm6.4}.  Now let $k$ be prime.   Recall Theorem \ref{Su1Thm5.6};  that equation  now becomes 
$$\sum_{\lambda\vdash n: \lambda_i=1, k} p_\lambda
=H[\sum_{m\geq 1} \text{ch}\, (\exp(\tfrac{2\pi i k}{n})\uparrow_{C_n}^{S_n})]\vert_{{\rm\, deg\,} n},$$ and thus the left-hand side 
is precisely $p^T$ for $k$ prime and $T=\{1,k\}.$ But $H-1$ is invertible with respect to plethysm, so  $H[F]=H[G]$ if and only if $F=G.$ Hence $f_n^T$ must coincide with $\ell_n^{(k)}$.
%%%%%%%%%%%%%%%%%%%%%%%%%%
%%%%%%%%%%%%%%%%%%%%%%%%%%%%%
\end{proof}

Computations  indicate  that 
\begin{conj}  $f_n^{\{1,k\}}$ is Schur positive for $k=2$  and for all odd $k\geq 3.$  (This is trivially true if $k=1.$)
\end{conj}

When $k$ is even and not equal to 2, this fails.  For instance, if $n=k=4m,$ it is easy to see that $Lie_{4m}+p_{4m}$ contains the sign representation with coefficient $(-1).$ However 
we have $H[F^{\{1,k\}}]=(1-p_1)^{-1}(1-p_k)^{-1},$ which we know to be Schur positive from \cite[Proposition 4.23]{Su1}.  This example shows that it is not always possible to write a Schur positive sum of power sums as a symmetrised module over a sequence of true $S_n$-modules, since $Lie_k+p_k$ fails to be Schur positive when $k$ is even.

\begin{prop}\label{prop6.8} Let $k\geq 2.$ Then $\omega (E[F^{\{1,k\}}])=(1-p_1)^{-1}(1-(-1)^{k-1}p_k)^{-1} (1+p_2)(1+p_{2k}).$  If $k$ is prime, this is Schur positive.
\end{prop}
\begin{proof}  The plethysm $E[F^{\{1,k\}}]$ is most easily calculated by using Part (3) of Proposition \ref{Pleth1} in the next section.  Since the series $F^{\{1,k\}}$ is Schur positive when $k$ is prime, the claim follows.\end{proof}

The three propositions that follow are also clear from Theorem \ref{thm6.4}.

\begin{prop}\label{prop6.9}  Let $k\geq 2$ and $T=\{n:n\leq k\}.$ Then 
$\prod_{n=1}^k (1-p_n)^{-1}=H[\sum_{n} f_n^T],$
where
\begin{equation}\label{eqn6.12} f_n^T=\sum_{\stackrel{m=1}{m|n}}^k Lie_{\frac{n}{m}}[p_m].\end{equation}
\end{prop}

\begin{cor}\label{cor6.10}   Let $T=\{n:n\leq k\}, k\geq 2.$  If $n$ is prime, or $n\le k,$ or $n>k$ and $n$ is such that the greatest proper divisor of $n$ is at most $k,$ then $f_n^T$ is Schur positive.
\end{cor}
\begin{proof} If $n$ is prime it is easy to see that 
$f_n^T=\begin{cases} Lie_n, & n>k\\
                                 Lie_n+p_n, & \mathrm{otherwise.}
\end{cases}$
Since in this case, $\mu(n)=-1$ and $\phi(n)=n-1,$ and thus $Lie_n=\frac{1}{n}(p_1^n-p_n),$ we conclude that 
$Lie_n+p_n=\frac{1}{n}(p_1^n +\phi(n)p_n) =Lie_n^{\mathcal{P}}.$

If $n\le k$ then the sum in equation (\ref{eqn6.12}) ranges over all divisors of $n$ and hence by (\ref{prop6.6eqn1}) we have $f_n^T=C\!onj_n.$ 

If $n>k$ and  the largest proper divisor of $n$ is at most $k,$ we have $f_n^T=C\!onj_n-p_n,$ again by (\ref{prop6.6eqn1}).  But it is well known that 
$p_n=\sum_{r\geq 0} (-1)^r s_{(n-r, 1^r)}.$  Also by a  result of 
\cite{Sw}, $C\!onj_n$ contains all hooks except for the following: $(n-1,1)$ for all $n\geq 2,$ 
$(2, 1^{n-2})$ for odd $n\geq 3$, and $(1^n)$ for even $n.$ In all three cases, the corresponding Schur function appears with coefficient $-1$ in $p_n,$ hence with coefficient  
$+1$ in $C\!onj_n-p_n.$ This finishes the proof. \end{proof}

\begin{conj} (See also \cite[Conjecture 1]{Su1}.) $f_n^{\{1,\ldots,k\}}$ is Schur positive for all $n$ and $k,$ and hence so is $\prod_{n=1}^k (1-p_n)^{-1}.$
\end{conj}

\begin{prop}\label{prop6.11} Let $k\geq 2$ and $T=\{n: n|k\}.$ Then 
\begin{equation}\label{eqn6.13} \prod_{n|k} (1-p_n)^{-1} = H[\sum_n f_n^T],
\quad \text{where}\quad f_n^T=\sum_{m|(k,n)} Lie_{\frac{n}{m}}[p_m]
\end{equation}
\end{prop}

 In particular from Theorem \ref{Su1Thm5.6} we immediately obtain 
 (since $H-1$ is invertible with respect to plethysm) two corollaries:
\begin{cor}\label{cor6.12} $$\ell_n^{(k)} ={\rm ch } \exp(2i\pi\cdot k/n)\uparrow_{C_n}^{S_n}= \sum_{m|(k,n)} Lie_{\frac{n}{m}}[p_m],$$
and hence $f_n^T$ is Schur positive when $T$ is the set of all divisors of $k.$
\end{cor}

\begin{cor}\label{cor6.13} We have the following decomposition of the regular representation into virtual representations:
\begin{equation}\label{eqn6.15} p_1^n=\sum_{k=1}^n \sum_{m|(k,n)} Lie_{\frac{n}{m}}[p_m]=\sum_{d|n}d\, Lie_d[p_{\frac{n}{d}}].
\end{equation}
\end{cor}
\begin{proof} This follows from the decomposition (see   \cite[Theorem 8.8]{R}) 
$$ p_1^n=\sum_{k=1}^n \ell_n^{(k)}$$
and the preceding corollary, because the first sum can be rewritten as $$\sum_{m|n} \sum_{\stackrel{r=1}{k=rm\leq n}}^{ \frac{m}{n}} Lie_{\frac{m}{n}}[p_m]=\sum_{m|n}\frac{m}{n}\ Lie_{\frac{m}{n}}[p_m].$$
One can also derive equation (\ref{eqn6.15}) directly by using the expansion of $Lie_n$ into power sums.
\end{proof}
\begin{prop}\label{prop6.14} Let $T=\{n:n\equiv 1\,\mathrm{mod}\, k\}.$
Then 
$$\prod_{n\equiv 1\,\mathrm{mod}\, k} (1-p_n)^{-1}=H[\sum_{n} f_n^T]\quad \text{for} \quad
f_n^T=\sum_{\stackrel{m\equiv 1\,\mathrm{mod}\, k}{ m|n}} Lie_{\frac{n}{m}}[p_m].$$
\end{prop}

After seeing a preprint  of \cite{Su1}, Richard Stanley made the following conjecture, and verified it for $n\leq 24$ and $k\leq 6$.
\begin{conj}\label{RPS2015} (R. Stanley, 2015)  $\prod_{n\equiv 1\,\mathrm{mod}\, k} (1-p_n)^{-1}$ is Schur positive for all $k.$ 
\end{conj}

As noted in equation (1) of Corollary \ref{Spositivity2} and other places, Conjecture~\ref{RPS2015} holds  for $k=2.$  We  have 
$$ \prod_{n\equiv 1\,\mathrm{mod}\, 2} (1-p_n)^{-1} =E[Conj]=H[Lie^{\overline{{(2)}}}].$$ In this case, writing $p^{\mathrm{ odd}}$ for 
$\sum_{n\, \mathrm{odd}} p_n, $ we have the identity
$$ p^{\mathrm {odd}}[Lie]=Lie^{\overline{{(2)}}},$$ and hence:

\begin{prop}\label{prop6.15}  $\sum_{\stackrel{m\,\mathrm{odd}}{ m|n}} Lie_{\frac{n}{m}}[p_m]=p^{\mathrm {odd}}[Lie]\vert_{{\rm deg\ } n}$ is  Schur positive; it is the representation   $Lie_n^{\overline{{(2)}}}=\exp(2i\pi\ell/n)\uparrow_{C_n}^{S_n},$ where 
$n=2^k\cdot\ell$ and $\ell$ is odd.
\end{prop}
\begin{proof} This is clear since the symmetric powers of the two modules coincide, both being equal to  $\prod_{n\, \mathrm{ odd}}(1-p_n)^{-1}.$
\end{proof}

Proposition \ref{prop6.6} gives us several different ways to decompose $C\!onj_n$, which we collect in the following:

\begin{thm}\label{thm6.16} 
 For any prime $q,$ we have 
\begin{equation}\label{eqn6.17} \sum_n C\!onj_n=\sum_{\stackrel{n}{ q \text{ does not divide } n}} p_n[Lie^{(q)}],
\end{equation}
and hence the sum on the right is Schur positive.
In fact for any positive integer $q$ we have 
\begin{equation}\label{eqn6.18} \sum_n C\!onj_n=\sum_{\stackrel{n}{q \text{ does not divide } n}} p_n[\sum_{k\geq 0} Lie[p_{q^k}]].
\end{equation}

\end{thm}

\begin{proof} We start with equation (\ref{prop6.6eqn1}) of Proposition \ref{prop6.6}, and use the fact that 
$(p_1-p_q)$ and $\sum_{k\ge 0} p_{q^k}$ are plethystic inverses (see Proposition~\ref{Pleth1} in the next section).
We have, by associativity of plethysm,
\begin{align*} \sum_{n\geq 1} C\!onj_n&= \sum_{n\geq 1} p_n[Lie]=\sum_{n\geq 1} (p_n[p_1-p_q] )[\sum_{k\geq 0} p_{q^k}[Lie]]\\
&=\sum_{n\geq 1} (p_n-p_{nq} )[\sum_{k\geq 0} p_{q^k}[Lie]]=(\sum_{n\geq 1} p_n -\sum_{n\geq 1} p_{nq})[\sum_{k\geq 0} Lie[p_{q^k}]];
\end{align*}
invoking Theorem \ref{thm6.4}, this establishes equation (\ref{eqn6.18}).\end{proof}

\begin{rk}  In   { \cite{SuLieSup2}}, it was conjectured that $Lie_n^{(2)}\uparrow_{S_n}^{S_{n+1}}-Lie_{n+1}^{(2)}$ is a true $S_{n+1}$ module which lifts $Lie_n^{(2)},$ if $n$ is not a power of 2.
One can ask if this holds for the odd primes $q.$

If $n$ is a power of an odd prime $q,$ it follows that 
$Lie_n^{(q)}={\rm ch\,} \textbf{ 1}\uparrow_{C_n}^{S_n}= C\!onj_n,$ while $Lie_{n-1}^{(q)}=Lie_n$ since $n-1\equiv q-1 \mod q,$ so that $n-1$ is relatively prime to $q.$ Now $Lie_n$ does not contain the sign representation (for $n\neq 2$) and never contains the trivial representation; however both appear once in the conjugacy action on the $n$-cycles, for odd $n,$ and hence both  appear with multiplicity $(-1)$ in $p_1 Lie_{n-1}^{(q)}- Lie_n^{(q)}.$ 

For $q=3$ we have verified, up to $n=32,$ that $p_1 Lie_{n-1}^{(3)}-Lie_n^{(3)}$ is Schur positive except for $n=3,6,9,10, 18, 27.$

For $q=5, n\leq 32,$ $p_1 Lie_{n-1}^{(5)}-Lie_n^{(5)}$ is Schur positive except for $n=5,6,10, 25, 26.$
\end{rk}

%%%%%%%%%%%%%%%%%%%%%%%%%%%%%%

\section{A compendium of plethystic inverses}\label{SecPleth}
This section contains some useful plethystic identities and manipulations.

\begin{prop}\label{Pleth1} We have 
\begin{enumerate}
\item For $q\geq 2,$ $p_1-p_q$ and $\sum_{k\geq 0}p_{q^k}$ are plethystic inverses. 
\item For $q\geq 2,$  $p_1+p_q$ and $\sum_{k\geq 0}(-1)^k p_{q^k}$ are plethystic inverses. 
\item For $q\ge 2, H[p_1-p_q]=\dfrac{H}{H[p_q]}.$
\item $\dfrac{H}{H[p_2]}=H[p_1-p_2]=E,$  and hence $(H-1)[p_1-p_2]=E-1.$
\item Suppose $F=\sum_{n\geq 1} f_n$ and $G=1+\sum_{n\geq 1} g_n$ are formal series of symmetric functions,  with $f_n, g_n$ being of homogeneous degree $n.$ Then $H[F]=G\iff G
=(\tfrac{H}{H[p_q]})[\sum_{k\ge 0} F[p_{q^k}]].$  In particular 
\[H[F]=G\iff  G=E[\sum_{k\ge 0} F[p_{2^k}]].\]
\end{enumerate}
\end{prop}
\begin{proof} Part (1) follows by calculating 
$$(p_1-p_q)[\sum_{k\geq 0}p_{q^k}]=\sum_{k\geq 0}p_{q^k}-
\sum_{k\geq 0}p_q[p_{q^k}]
=\sum_{k\geq 0}p_{q^k}-\sum_{k\geq 0}p_{q^{k+1}}=p_1.$$
Part (2) follows in the same manner.
For Part (3), use the exponential generating function ~\eqref{defHE} for $H.$ 
Then 
$$H[p_1-p_q]=\exp\sum_{i\geq 0} \frac{p_i}{i}[p_1-p_q]
=\exp\sum_{i\geq 0} \tfrac{1}{i}(p_1-p_q)[{p_i}]
=\frac{\exp\sum_{i\geq 0} \frac{p_i}{i}}{\exp\sum_{i\geq 0}\frac{p_{qi}}{i}}=\frac{H}{H[p_q]}.$$
Now Part (4) follows by observing that 
$\sum_{i\geq 0} \frac{p_i}{i}-\sum_{i\geq 0}\frac{p_{2i}}{i}
=\sum_{i\geq 0} (-1)^{i-1}\frac{p_i}{i}, $ and hence $H[p_1-p_2]=E, $ using ~\eqref{defHE} again.
 
 For Part (5), use Part (1) to rewrite 
 $H[F]=(H[p_1-p_q][\sum_{k\geq 0}p_{q^k}])[F],$ and then use  Part (3) and the fact that plethysm is associative.
 \end{proof}

\begin{prop}\label{Pleth3} Let $F$ and $G$ be two series of symmetric functions with no constant term.  Then  the following are equivalent:
\[H[F]=E[G]\iff E^{\pm}[F]=H^{\pm}[G]\iff F=G-G[p_2]\iff G=\sum_{k\geq 0} F[p_{2^k}].\]
In particular, any  symmetric power of modules is also expressible as an exterior power of modules.  Also, if $F$ is Schur positive, then so is $G-G[p_2].$
\end{prop}
\begin{proof}  The equivalence of the first two statements follows  from Lemma \ref{pm}.  We have $(H-1)[F]=(E-1)[G],$ and hence, using associativity of plethysm, 
$F=((H-1)^{\langle-1\rangle}[E-1])[G],$
$G=((E-1)^{\langle-1\rangle}[H-1])[F].$  

Using Proposition~\ref{Pleth1}, we have 
$H[F]=E[G]=(H[p_1-p_2])[G]$ and hence $F= (p_1-p_2)[G]=G-G[p_2].$

Similarly, using the plethystic inverse of $p_1-p_2,$ we have 
$H[F]=H[p_1[F]]=(H[p_1-p_2])([\sum_{k\geq 0} p_{2^k}[F]])
=E[\sum_{k\geq 0} p_{2^k}[F]]=E[\sum_{k\geq 0} F[p_{2^k}]],$ and hence $G=\sum_{k\geq 0} F[p_{2^k}].$  In both cases we use the fact that $H-1, E-1$ are invertible with respect to plethysm.
\end{proof}

Equations $(A)$ and $(B)$ in the following example illustrate  Proposition \ref{Pleth3} and Corollary \ref{Spositivity2}.
\begin{ex}\label{DualityEx}  Let $S=\{2\}.$ 
\begin{equation*}
(1-p_1)^{-1}=H[Lie^\emptyset]=H[Lie]=E[Lie^{S}]=E[Lie^{(2)}],\qquad (A)
\end{equation*}
 from equations  (\ref{SymLS}) and Part (1) of Theorem \ref{PlInv-E}.

 Note in particular that Proposition \ref{Pleth3} gives $Lie=Lie^{(2)}-Lie^{(2)}[p_2].$

\vskip .05in
Now let $S$ be the set of odd primes. Then
\begin{equation*}
\prod_{ n {\rm \, odd}} (1-p_n)^{-1}=H[Lie^{\overline{\{2\}}}]=E[Lie^{\bar{\emptyset}}]
=E[Conj], \qquad (B)
\end{equation*}
 from Theorem \ref{SymLS} and Proposition \ref{ExtLieConj}. 
In this case, the preceding proposition asserts that  $Lie^{\overline{\{2\}}}=C\!onj-C\!onj[p_2];$ in particular the latter is Schur positive.
\end{ex}

For arbitrary $q$ we have:
\begin{prop}\label{Pleth4} Let $q\ge 2,$ and let $g_n=\dfrac{1}{n}\sum_{d|n} \psi(d) p_d^{\tfrac{n}{d}},$  $G=\sum_{n\geq 1} g_n$ and  
$F\!=\!(p_1\pm p_q)[G]=\sum_{n\geq 1} f_n.$ 

Then
$f_n=\dfrac{1}{n}\sum_{d|n} \bar{\psi}(d) p_d^{\frac{n}{d}},$
where 
$\bar\psi(d)=
\begin{cases} \psi(d)\pm q\psi(\frac{d}{q}), &q|d,\\
                    \psi(d), & \mathrm{otherwise}.
\end{cases}$
\end{prop}
\begin{proof} We will do the case $F=(p_1-p_q)[G],$ since the other case is identical. Note that $F=(p_1-p_q)[G]= G-p_q[G]=G-G[p_q].$ 
It is easiest to use Proposition \ref{Su1Prop3.1}:
 \begin{align*}
F &=\log \prod_{d\geq 1} (1- p_d)^{-\frac{\psi(d)}{d}}
-\log \prod_{d\geq 1} (1- p_{dq})^{-\frac{\psi(d)}{d}}\\
&=\log \dfrac{\prod_{d\geq 1} (1- p_d)^{-\frac{\psi(d)}{d}}}
{\prod_{d\geq 1} (1- p_{dq})^{-\frac{\psi(d)}{d}}}
=\log \left(\prod_{d\geq 1,\, q \not{|} d} (1- p_d)^{-\frac{\psi(d)}{d}} \prod_{m\geq 1} (1- p_{mq})^{-\frac{\psi(mq)}{mq}+\frac{\psi(m)}{m}}\right)
\end{align*}
\end{proof}

We collect the plethystic inverses from this paper and from the literature that are of  representation-theoretic  and homological significance, and indeed, were first derived in those contexts. If $g$ is a symmetric function without constant term, and with nonzero term of degree 1, we denote by $g^{\langle-1\rangle} $ the plethystic inverse of $g. $ Equations (\ref{Pleth8a}) and (\ref{Pleth8b}) first appeared in \cite{CHR} and \cite{CHS} respectively.   
\begin{prop}\label{Pleth8}   The following pairs are plethystic inverses:
\begin{equation}\label{Pleth8a}
 \dfrac{p_1}{1+p_1} {\rm\ and\ } \dfrac{p_1}{1-p_1} ;\end{equation}
\begin{equation}\label{Pleth8b}\sum_{n\geq 1} g(n) p_n {\rm\ and\ }\sum_{n \geq 1}g(n) \mu(n) p_n, \end{equation}
for any function $g(n)$ defined on the positive integers, such that $g(mn)=g(m)g(n);$ 
\begin{equation}\label{Pleth8c}\sum_{n\text{ odd}} g(n) p_n {\rm\ and\ }\sum_{n\text{ odd}} g(n) \mu(n) p_n, \end{equation}
for any function $g(n)$ defined on the positive integers, such that $g(mn)=g(m)g(n);$ 
 \begin{equation}\label{Cadogan} \sum_{i\geq 1} (-1)^{i-1} \omega(Lie_i) \text{ and } H-1; \end{equation}  
 \begin{equation}\label{Sheila} \sum_{i\geq 1} (-1)^{i-1} \omega(Lie_i^{(2)})
 \text{ and } E-1; \end{equation}
\begin{equation} \label{Pleth8d}
\sum_{n\geq 1} Lie_n  {\rm\ and\ }\dfrac{H-1}{H}=\sum_{n\geq 1} (-1)^{n-1} e_n;
\end{equation}
\begin{equation} \label{Pleth8e}
\sum_{n\geq 1} Lie_n^{(2)} {\rm\ and\ } \dfrac{E-1}{E}=\sum_{n\geq 1} (-1)^{n-1} h_n,
\end{equation}
and hence $\omega(Lie_n^{\langle-1\rangle})
=(Lie_n^{(2)})^{\langle-1\rangle}.$

\begin{equation}\label{Pleth8f}
\sum_{n\equiv 1 {\rm\, mod\, }k} h_n {\rm\ and\ }\sum_{n\geq 0} (-1)^n \beta_{nk+1}; 
\end{equation}
here $\beta_{nk+1} $ is the Frobenius characteristic of the homology representation of $S_{nk+1},$   for the pure Cohen-Macaulay set partition poset of $nk+1$ with block sizes congruent to 1 modulo $k$   \cite[Theorem 4.7]{CHR}, from which one deduces that
\begin{equation}\label{Pleth8f2} 
\sum_{n\equiv 1 {\rm\, mod\, }k} e_n {\rm\ and\ }\sum_{n\geq 0} (-1)^n \omega(\beta_{nk+1}) ; 
\end{equation}
when $k$ is \textbf{even}, and in particular when $\mathbf{k=2},$ and  $\beta_{nk+1} $ is as above.
\begin{equation}  \label{Pleth8g}
\sum_{n\geq 0} \eta_{2n+1} {\rm\ and\ }  \sum_{m\geq 1} (-1)^{m-1} \dfrac{\partial}{\partial p_1} \beta_{2m};
\end{equation}
here $\eta_{n}$ is the Frobenius characteristic of the $S_n$-action on the multilinear component of the free Jordan algebra on $n$ generators \cite[Proposition 3.5]{CHS}, while $\beta_{2m}$ is the homology representation of $S_{2m}$ for the set partition poset of ${2m}$ with even block sizes \cite[Theorem 4.7]{CHR}, from which one deduces that

\begin{equation} \label{Pleth8g2}
\sum_{n\geq 0} \omega(\eta_{2n+1}) {\rm\ and\ }  \sum_{m\geq 1} (-1)^{m-1} \dfrac{\partial}{\partial p_1} \omega(\beta_{2m}).
\end{equation}

\end{prop}

\begin{proof}  Using the property of plethysm that  $p_n[f]=f[p_n]$ for all symmetric functions $f,$ the first three pairs are easily checked to be plethystic inverses  by direct computation and the defining property of the M\"obius function in the divisor lattice.

Equation (\ref{Pleth8d}) is well known  as a consequence of the acyclicity of the (sign-twisted) Whitney homology for the partition lattice e.g.\cite[Theorem 1.8 and Remark 1.8.1]{Su0}.  The derivation we give here shows that it is  in fact equivalent to  Thrall's decomposition  of the regular representation, namely  $H[Lie]=(1-p_1)^{-1}.$ Applying (1) of Lemma \ref{pm} to the latter identity immediately gives (\ref{Pleth8d}).  Similarly, (\ref{Pleth8e}) is equivalent to our decomposition $E[Lie^{(2)}]=(1-p_1)^{-1}$ of the regular representation, see (the first equation in) Theorem \ref{PlInv-E},  as can be seen by invoking (2) of Lemma \ref{pm}.

The inverse pair (\ref{Pleth8f}) is a result of \cite[Theorem 4.7 (a), p. 297]{CHR}.   It  also reflects the acyclicity of Whitney homology for the poset of partitions of a set of size $nk+1$ into blocks of size congruent to 1 modulo $k$  (see \cite[Example 3.7]{SuJer}). 
The pair (\ref{Pleth8f2}) follows by applying the involution $\omega$ to (\ref{Pleth8f}), since, when $k$ is even, all the symmetric functions involved have odd degree, and we have $\omega(f[g])
=\omega(f)[\omega[g]$ when $g$ is homogeneous of odd degree.  

Finally, the inverse pair (\ref{Pleth8g}) is a result of \cite[Proposition 3.5]{CHS}, with ~\eqref{Pleth8g2} following as in the preceding case.
\end{proof}

There are also interesting expressions for the plethystic inverse of the sum of the odd Lie representations $Lie_{2m+1}$, and of the alternating sum.  See \cite{SuOddLie}. 

We conclude with two questions.

\begin{qn} Is there a formula describing, for each partition  $\lambda,$ the irreducible decomposition of $H_\lambda[Lie_n^{S}],$ the higher $Lie_n^{S}$-module for arbitrary subsets $S$ of primes?   This is open for all $S,$   including  the classical $Lie$ case (Thrall's problem) when $S=\emptyset.$   Note that Theorem~\ref{Foulkes} gives a nice combinatorial formula, in terms of standard Young tableaux and the major index, for the irreducible decomposition of $Lie_n^S$ itself.
\end{qn}

\begin{qn}  A beautiful theorem of Gessel and Reutenauer states that the higher Lie module $H_\lambda[Lie_n]$ is the fundamental quasi-symmetric function corresponding to the conjugacy class indexed by $\lambda$ \cite[Theorem 3.6]{GR}. Is there an analogous theory  corresponding to $H_\lambda[Lie_n^{(q)}]$ or $E_\lambda[Lie_n^{(q)}]$  when $q$ is prime? The case $q=2$ is of particular interest in view of the curious properties of $Lie_n^{(2)}$ derived in \cite{SuLieSup2}.
\end{qn}

%%%%%%%%%%%%%%%%%%%%%%%%%%%

% If using bibtex with separate .bib file
%\bibliographystyle{amsplain.bst}
%\bibliography{rps70.bib}

\bibliographystyle{amsplain.bst}

\end{document}